\documentclass[letterpaper,12pt]{amsart}

\usepackage{amsfonts,amsmath,amssymb,amsthm}

\newcommand{\norm}[2]{||#1||_{#2}}
\newcommand{\pnorm}[2]{\rho_{#2}(#1)}
\newcommand{\N}{\ensuremath{\mathbb{N}}}
\newcommand{\no}{||\cdot||}

\theoremstyle{definition}
	\newtheorem{defi}{Definition}[section]
\theoremstyle{plain}
	\newtheorem{lem}[defi]{Lemma}
	\newtheorem{thm}[defi]{Theorem}
	\newtheorem{prop}[defi]{Proposition}

\title{A norm for Tsirelson's Banach Space}
\author{Diana Ojeda-Aristizabal}
\address{Department of Mathematics, Cornell University, Ithaca,NY, 14853-4201}
\thanks{The research of the author presented in this paper was partially supported by NSF
grant DMS\textendash 0757507. Any opinions, findings, and conclusions or recommendations
expressed in this article are those of the author and do not necessarily reflect the
views of the National Science Foundation.}

\begin{document}

\bibliographystyle{plain}

\begin{abstract}
 	We give an expression for the norm of the space constructed by Tsirelson. The implicit equation satisfied by this norm is dual to the implicit equation for the norm of the dual of Tsirelson space given by Figiel and Johnson. The expression can be modified to give the norm of the dual of any mixed Tsirelson space. In particular, our results can be adapted to give the norm for the dual of Schlumprecht space.
\end{abstract}
\maketitle

\section{Introduction}
	In the 1970s Tsirelson \cite{Tsi} constructed a space with no isomorphic copies of $c_0$ or $\ell_p$, ($1\leq p<\infty$). The special properties of Tsirelson's space $T$ derive from certain saturation properties of the unit ball. The original construction of the space is geometric: one defines a subset $V$ of $\ell_{\infty}$ with certain properties and takes $T$ to be the linear span of $V$ with the norm that makes $V$ be the unit ball. There is no expression for the norm of $T$, which makes it difficult to study the space.\\

	Later, Figiel and Johnson \cite{Tsi_Figiel} gave the following implicit expression for the norm of $T^*$, the dual of $T$:
	\[
	\norm{x}{}=\max\left\{\norm{x}{\infty},\frac{1}{2}\max\left\{\sum_{i=1}^k \norm{E_ix}{}: k\in\N, k\leq E_1<\cdots<E_k\right\}\right\}.
	\]

	 It is the dual of the original  Tsirelson space that came to be known as Tsirelson space and it is denoted in the literature by $T$. Since we are analyzing the original construction due to Tsirelson, we shall call the space he constructed $T$, the space $F$ will be the completion of $c_{00}$ with respect to the norm given by Figiel and Johnson.\\

	Many results about $T^*$ followed and the properties of $T^*$ were widely studied (see \cite{Tsi_Casazza}). Besides studying the properties of $F$, many Tsirelson-type spaces were studied, that is, spaces whose norms are given by some variation of the expression found by Figiel and Johnson. Schlumprecht space, constructed in \cite{Schlum}, is an example of such a space. In this paper we will give an expression for the norm of the original Tsirelson space. We will prove that there is a norm $||\cdot||$ on $c_{00}$ that satisfies the following implicit equation
	\begin{eqnarray*}
	||x||&=& \min\{2 \min\{\underset{1\leq i\leq k}{\max}\norm{E_ix}{}:k\leq E_1<\cdots <E_k, x=\sum_1^k E_i x\},\\
		&&\inf\{\norm{y}{}+\norm{z}{}:x=y+z, \mathrm{supp} (y)\subseteq \mathrm{supp} (x)\}\}.
	\end{eqnarray*}
	In fact, we shall prove that the norm of the original Tsirelson space is maximum (in the point-wise sense) among the norms satisfying the implicit equation above and such that $||e_i||=1$ for every vector $e_i$ in the standard basis of $c_{00}$. Note that, as opposed to the implicit equation given by Fiegel and Johnson, this expression doesn't allow us to calculate the norm of finitely supported vectors inductively in the cardinality of the support.\\
The expression for this norm can be adapted for the dual of any mixed Tsirelson space. In particular one can get an expression for the norm of the dual of Schlumprecht space.\\

	In section \ref{prelim} we introduce the notation and the results from Banach space theory we shall use. In section \ref{norm_def} we define a norm on $c_{00}$ and prove that Tsirelson's space $T$ is the completion of $c_{00}$ with respect to this norm. In section \ref{mixed} we will give an expression for the norm of the dual of any mixed Tsirelson space and show how our results from section \ref{norm_def} can be adapted to this case.

\section{Notation and preliminaries}\label{prelim}
	
	A standard reference on Schauder bases is \cite{Lind_Tza}, we follow the notation therein. The space $c_{00}$ is the space of finitely supported sequences of real numbers with the sup norm, the sequence $(e_n)_n$ denotes the standard Schauder basis for $c_{00}$.\\
	
We shall denote  the original Tsirelson space by $T$. Figiel and Johnson \cite{Tsi_Figiel} proved that there is a unique norm $\norm{\cdot}{F}$ on $c_{00}$ that satisfies the implicit equation
	\small
	\[
	\norm{x}{}=\max\left\{\norm{x}{\infty},\frac{1}{2}\max\left\{\sum_{i=1}^k \norm{E_ix}{}: k\in\N, k\leq E_1<\cdots<E_k\right\}\right\},
	\]
	and such that $\norm{e_i}{F}=1$ for all $i\in\N$. This follows from the fact that we can calculate the norm of a given vector in $c_{00}$ by calculating the norm of vectors with strictly smaller supports.\\
	 Let $F$ be the completion of $c_{00}$ with respect to this norm. The sequence $(e_n)_n$ is a Schauder basis for $F$ and $F$ is isometrically isomorphic to the dual of $T$. The norm $\norm{\cdot}{F}$ can be obtained as a limit of norms in the following way, for $x\in c_{00}$ define
	\small
	\begin{eqnarray*}
		\norm{x}{0}&=&\norm{x}{\infty}\\
		\norm{x}{n+1}&=& \max\left\{\norm{x}{n},\frac{1}{2}\max\left\{\sum_{i=1}^k \norm{E_ix}{n}: k\in\N, k\leq E_1<\cdots<E_k\right\}\right\}
	\end{eqnarray*}
	
	It can be proved that if we let $\norm{x}{F}=\lim_{n\to\infty}\norm{x}{n}$, then $\norm{\cdot}{F}$ satisfies the implicit equation above. \\

	In the course of our proof, we will use the Bipolar Theorem, its proof can be found in \cite[Section 3.4]{B_space_fund}. For this we need some additional notation.\\
	Let $X$ be a Banach space, and for $A\subset X$, $B\subset X^*$ we define
	\begin{eqnarray*}
	 	A^{\circ} &=& \{f\in X^*: \mbox{for all }x\in A, |f(x)|\leq 1\}\\
		B^{\circ} &=& \{x\in X: \mbox{for all }f\in B, |f(x)|\leq 1\}.
	\end{eqnarray*}
We shall need the following instance of the Bipolar Theorem:
	\begin{thm}
	 	(Alaouglu, Banach) Let $X$ be a Banach space. For every $A\subset X^*$, $A^{\circ\circ}$ is the weak*- closure of the convex hull of $A\cup\{0\}$.
	\end{thm}
	
Given a Banach space $X$ with a shrinking basis $(x_n)_n$, we can identify $f\in X^*$ with the sequence of scalars $(a_n)_n$ such that $f=\sum a_n x_n^*$. It is clear that if $(f^m)_m$ is a sequence in $X^*$ with $f^m=\sum a_n^m x_n^*$ and such that $f^m$ converges to $f$ with respect to the weak*-topology, then $f=\sum b_n x_n^*$ where $\lim_{m\to\infty}a_n^m=b_n$. The following proposition shows that in fact, the weak*- topology and the topology of pointwise convergence coincide in the unit ball of $X^*$.

\begin{prop}\label{weak_pw}
 	let $X$ be a Banach space with a shrinking basis $(x_n)_n$. Let $(f^m)_m\subset B_{X^*}$ where $f^m=\sum a_n^m e_n^*$, be such that for each $n\in\N$, $\lim_{m\to\infty}a_n^m=b_n$. Let $f=\sum b_n e_n^*\in X^*$ then $f^m$ converges to $f$ with respect to the weak*-topology.
\end{prop}

\section{Definition of the norm}\label{norm_def}	
	As we saw in the previous section, the norm for $T^*$ can be obtained by taking the limit of a sequence of norms. To get an expression for the norm of $T$ we shall take the limit of a sequence of positive scalar functions on $c_{00}$. In this case each scalar function is not a norm but the sequence is defined in such a way that the pointwise limit is a norm on $c_{00}$,\\
	 
	\begin{defi}\label{norm} For $x\in c_{00}$ let
	\begin{eqnarray*}
		\rho_0(x)&=&||x||_{\ell_1}\\
		\rho_{n+1}(x)&=& \min\{2 \min\{\underset{1\leq i\leq k}{\max}\pnorm{E_ix}{n}:k\leq E_1<\cdots <E_k, x=\sum_{i=1}^k E_i x\},\\
		&&\inf\{\rho_n(w^1)+\rho_n(w^2):x=w^1+w^2\}\}\\
		||x||&=&\lim_{n\to\infty}\rho_n(x)
	\end{eqnarray*}
	\end{defi}
	\begin{lem}\label{isnorm}
		The function $||\cdot||$ defines a norm on $c_{00}$ and whenever $x=(x_n)$, $y=(y_n)\in c_{00}$ are such that $|y_n|\leq |x_n|$, we have $||y||\leq ||x||$.
	\end{lem}
	\begin{proof}
		Let $x=(x_n), y=(y_n)\in c_{00}$ be such that $|y_n|\leq |x_n|$. We prove by induction on $n$ that $\pnorm{y}{n}\leq\pnorm{x}{n}$. The base case is clear so suppose the inequality holds for $n$. Let $k\leq E_1<\cdots <E_k$ be such that $x=\sum E_i x$. Note that since $\mathrm{supp}(y)\subseteq \mathrm{supp} (x)$, we have that $y=\sum E_i y$.\\
	By induction hypothesis $\underset{1\leq i\leq n}{\max}\pnorm{E_iy}{n}\leq \underset{1\leq i\leq n}{\max}\pnorm{E_ix}{n}$ so 
	\[
		\min\left\{\underset{1\leq i\leq n}{\max}\pnorm{E_iy}{n}:k\leq E_1<\cdots <E_k, y=\sum E_i y\right\}\leq \underset{1\leq i\leq n}{\max}\pnorm{E_ix}{n}.
	\]
	Since $k\leq E_1<\cdots <E_k$ was arbitrary, we have that 
	\[
		\pnorm{y}{n+1}\leq 2 \min\left\{\underset{1\leq i\leq n}{\max}\pnorm{E_ix}{n}:k\leq E_1<\cdots <E_k, x=\sum E_i x\right\}.
	\] 
	Let $x^1,x^2$ be such that $x=x^1+x^2$ , then we can find $y^1,y^2$ such that $y=y^1+y^2$ and $|y^i_n|\leq |x^i_n|$, hence 
	\[
		\pnorm{y}{n+1}\leq \inf\{\pnorm{w^1}{n}+\pnorm{w^2}{n}:x=w^1+w^2\}.
	\]
	It follows that $\pnorm{y}{n+1}\leq\pnorm{x}{n+1}$. Therefore for all $n\in\N$, $\pnorm{y}{n}\leq\pnorm{x}{n}$ and it follows that $||y||\leq||x||$.\\
	This monotonicity implies that $||\cdot ||$ is bounded below by the sup norm. Now we prove that $||\cdot||$ defines a norm.\\

	We summarize the properties of the sequence $(\rho_n)_n$ that will be used in the proof:
	\begin{itemize}
		\item [(i)] For all $x\in c_{00}$, $(\pnorm{x}{n})_n$ is monotone decreasing,\\
		\item [(ii)] For $x,y\in c_{00}$, $\pnorm{x+y}{n+1}\leq\pnorm{x}{n}+\pnorm{y}{n}$.
		\item [(iii)] For any $\lambda\in\mathbb{R}$ and $x\in c_{00}$, $\pnorm{\lambda x}{n}=|\lambda|\pnorm{x}{n}$.
	\end{itemize}
	The only non trivial property we must verify is the triangle inequality. Let $x,y\in c_{00}$, first we prove that for all $m\in\N$, 
	\begin{equation}
		||x+y||-||x||\leq\pnorm{y}{m}.\label{1}
	\end{equation}
	Let $m\in\N$, let $n>m$ then
	\[
		||x+y||\leq\pnorm{x+y}{n+1}\leq\pnorm{x}{n}+\pnorm{y}{n}\leq\pnorm{x}{n}+\pnorm{y}{m},
	\]
	therefore $||x+y||-\pnorm{y}{m}\leq\pnorm{x}{n}$ and this holds for all $n>m$ so 
	\[
		||x+y||-\pnorm{y}{m}\leq\underset{n>m}{\inf}\pnorm{x}{n}=||x||,
	\]
	hence $||x+y||-||x||\leq\pnorm{y}{m}$.\\
	Since \eqref{1} holds for all $m\in\N$, we have that $||x+y||-||x||\leq\underset{m}{\inf}\pnorm{y}{m}=||y||$. Hence $||\cdot||$ defines a norm on $c_{00}$.
	\end{proof}
We will now see that the norm $\no$ satisfies an implicit expression, dual to that found by Figiel and Johnson.
\begin{prop}\label{prop_min}
 	For any $x\in c_{00}$, we have that
	\begin{eqnarray*}
	||x||&=& \min\{2 \min\{\underset{1\leq i\leq k}{\max}\norm{E_ix}{}:k\leq E_1<\cdots <E_k, x=\sum E_i x\},\\
		&&\inf\{\norm{y}{}+\norm{z}{}:x=y+z, \mathrm{supp} (y)\subseteq \mathrm{supp} (x)\}\}. 
	\end{eqnarray*}
	Furthermore, if $||\cdot||'$ is a norm satisfying the implicit equation above and such that $||e_i||'=1$ for every vector $e_i$ in the standard basis of $c_{00}$, then for all $x\in c_{00}$, $||x||'\leq ||x||$.	
\end{prop}

\begin{proof}
	Since $\no$ satisfies the triangle inequality, $||x||=\inf\{\norm{y}{}+\norm{z}{}:x=y+z, \mathrm{supp} (y)\subseteq \mathrm{supp} (x)\}$. So we have to check that
	\[
		||x||\leq 2 \min\{\underset{1\leq i\leq k}{\max}\norm{E_ix}{}:k\leq E_1<\cdots <E_k, x=\sum E_i x\}.
	\]
 	Let $k\leq E_1<\cdots <E_k$ be such that $x=\sum E_i x$, and define 
	\[
		J=\{j\leq k:\underset{1\leq i\leq k}{\max}\norm{E_ix}{}=||E_jx||\}.
	\]
	Let $j_0\in J$ be such that for some cofinal $C\subseteq\N$ we have that for $n\in C$, $\underset{1\leq i\leq k}{\max}\pnorm{E_ix}{n}=\pnorm{E_{j_0}x}{n}$.\\
	Then for $n\in C$
	\[
		||x||\leq\pnorm{x}{n+1}\leq 2\underset{1\leq i\leq k}{\max}\pnorm{E_ix}{n}=2\pnorm{E_{j_0}x}{n}.
	\]
	So $||x||\leq 2\norm{E_{j_0}x}{}=2\underset{1\leq i\leq k}{\max}\norm{E_ix}{}$.\\
	Now suppose $||\cdot||'$ is a norm on $c_{00}$ that satisfies the implicit equation and such that $||e_i||'=1$ for every vector $e_i$ in the standard basis of $c_{00}$. It follows easily by induction that for all $n\in\N$ and all $x\in c_{00}$, $||x||'\leq\rho_n(x)$. Hence for all $x\in c_{00}$, $||x||'\leq ||x||$.
\end{proof}

Unlike the expression given by Figiel and Johnson, this implicit expression does not allow us to calculate the norm of a vector recursively on the cardinality of its support. This is because of the infimum term in each $\rho_n$, this term is necessary in order to have the triangle inequality hold in the limit.\\

Since $T$ is reflexive and $T^*=F$, the space $T$ is isometrically isomorphic to the dual of $F$. We will prove that the norm $\no$ we defined and the norm of $F^*$ coincide in $c_{00}$.\\

We observed before that the sequence $(e_n)_n$ is a Schauder basis for $F$, let $(e_n^*)_n$ be the corresponding coefficient functionals. We shall define a subset of $F^*$ that contains all the information needed to calculate the norm of a given vector $x\in F$. Let
	\begin{eqnarray*}
		V_0 &=& \{\pm e_k^*: k\in\N\}\\
		V_{n+1} &=& V_n\cup \left\{\frac{1}{2}(f_1+\cdots +f_k): k\in\N, k\leq f_1<\cdots<f_k, f_i\in V_n\right\}\\
		V &=& \bigcup_n V_n.
	\end{eqnarray*}

	\begin{prop}\label{det}
 		For every $x\in F$, $\norm{x}{F}=sup\{f(x):f\in V\}$.
	\end{prop}
	\begin{proof}
		Let $(\norm{\cdot}{F,n})_n$ be the sequence of norms defined by Figiel and Johnson. For each $n\in\N$ and $x\in c_{00}$ define $\tau_n(x)=\mbox{sup}\{f(x):f\in V_n\}$. It is easy to prove  by induction on $n$ that $\tau_n(x)=\norm{x}{F,n}$ for every $x\in c_{00}$. Therefore
		\[
 			\norm{x}{F}=\lim_{n\to\infty}\norm{x}{F,n}=\lim_{n\to\infty}\tau_n(x)=\mbox{sup}\{f(x):x\in V\}.
	\]
	\end{proof}

We are now ready to use the Bipolar Theorem to prove the following
	\begin{prop}\label{char_ball}
	 The unit ball of the dual of $F$ is the weak*-closure of the convex hull of $V\cup\{0\}$.
	\end{prop}
	\begin{proof}
	 	By proposition \ref{det}, $V^{\circ}=B_F$ and $V^{\circ\circ}=(B_F)^{\circ}= B_{F^*}$. Hence by the Bipolar theorem, we have that $B_{F^*}$ is the weak*- closure of the convex hull of $V\cup\{0\}$.
	\end{proof}
	In order to prove that $X$ is the dual of $F$, we need the following lemma:
	\begin{lem}
	 	For $x\in c_{00}$, if $\pnorm{x}{n}\leq 1$ then $x\in B_{F^*}$.
	\end{lem}
	\begin{proof}
	 	By propositions \ref{char_ball} and \ref{weak_pw}, $B_{F^*}$ is the closure of the convex hull of $V\cup\{0\}$ with respect to the topology of poinwise convergence. We shall use this to prove that $B_{F^*}$ has the following properties:
	\begin{itemize}
 		\item[(i)] The sequence $(e_n^*)$ is contained in $B_{F^*}$,
		\item[(ii)] if $f=(f_n)\in B_{F^*}$ and $g=(g_n)$ is such that $|g_n|\leq |f_n|$ for all $n\in\N$, then $g\in B_{F^*}$,
		\item[(iii)] if $f_1,\cdots,f_n\in B_{F^*}\cap c_{00}$ are such that $n\leq f_1<\cdots<f_n$, then $1/2(f_1+\cdots+f_n)\in B_{F^*}\cap c_{00}$.
	\end{itemize}
	This follows from the original construction of $T$. We include the proof for completeness, since we want to generalize our arguments to the dual of mixed Tsirelson spaces.\\
	Property (i) is clear by the definition of $V$. Note that the set $V$ has the closure property described in property (iii). Let $V'$ be the convex hull of $V\cup\{0\}$ then $V'$ has property (ii). Let $f_1,\cdots,f_n\in V'$ be such that $n\leq f_1<\cdots<f_n$ and let $f=1/2(f_1+\cdots +f_n)$. For each $i=1\leq n$, $f_i$ can be written as
	\[
 		f_i=\alpha_i^1g_i^1+\cdots+\alpha_i^mg_i^m
	\]
	for some $m\in\N$, some $g_i^j\in V\cup \{0\}$ and some non negative scalars $\alpha_i^j$ such that $\sum_s\alpha_i^s=1$ for all $i\leq n$. We may assume that $\mathrm{supp}(g_i^j)\subset\mathrm{supp}(f_i)$ for each $i,j$. Therefore for each choice of $s_1<\cdots<s_n\leq m$ we have that $n\leq g_1^{s_1}<g_2^{s_2}<\cdots<g_n^{s_n}$, so 
	\[
	 	\frac{1}{2}(g_1^{s_1}+\cdots+g_n^{s_n})\in V.
	\]
	Since $f$ can be written as a convex combination of vectors of this form, it follows that $f\in V'$. Hence $V'$ also has the closure property described in (iii). It easily follows that the closure of $V'$ also has properties (ii) and (iii).\\

We now prove by induction on $n$ that for $x\in c_{00}$, if $\pnorm{x}{n}\leq 1$ then $x\in B_{F^*}$. For the base case, assume that $||x||_{\ell_1}\leq 1$ and $x=(x_i)_i^k$ for some $k$, then $\sum_1^k |x_i|\leq 1$ so $x\in B_{F^*}$ by convexity of $B_{F^*}$.

	Now suppose that $\pnorm{x}{n+1}\leq 1$. If $\pnorm{x}{n+1}=2\underset{1\leq i\leq k}{\max}\pnorm{E_ix}{n}$ for some $k\leq E_1<\cdots <E_k$ such that $x=\sum E_i x$, then $\pnorm{2E_ix}{n}\leq 1$ for all $1\leq i\leq k$. By induction hypothesis this implies that $2 E_i x\in B_{F^*}$ for all $1\leq i\leq k$. Hence $x=\frac{1}{2}\sum 2E_ix\in B_{F^*}$ by property (iii).\\
	Now suppose that $\pnorm{x}{n+1}=\inf\{\pnorm{y}{n}+\pnorm{z}{n}:x=y+z\}$. For each $k\in\N$ let $y_k,z_k\in c_{00}$ be such that $x=y_k+z_k$ and $\pnorm{y_k}{n}+\pnorm{z_k}{n}\leq 1+1/k$. Then $u_k:= y_k/\rho_n(y_k),  v_k:=z_k/\rho_n(z_k)\in B_{F^*}$, by induction hypothesis. Since $$ \frac{x}{\rho_n(y_k)+\rho_n(z_k)}= \frac{\rho_n(y_k)}{\rho_n(y_k)+\rho_n(z_k)} u_k+  \frac{\rho_n(z_k)}{\rho_n(y_k)+\rho_n(z_k)}v_k,   $$
	and $B_{F^*}$ is convex, it follows that $ x/(\rho_n(y_k)+\rho_n(z_k))\in B_{F^*}$, or in other words, $$x\in  (\rho_n(y_k)+\rho_n(z_k)) B_{F^*}\subset (1+1/k)B_{F^*}$$ for every $k$. Hence, $x\in B_{F^*}$.

	\end{proof}
	Since $(e_n)_n$ is a shrinking basis of $F$, $(e_n^*)_n$ is a basis for $F^*$. Hence one can define on $c_{00}$ the $F^*$-norm,
$$\|(a_n)_n\|_{F^*}=\|\sum_{n}a_ne_n^*\|$$ for $(a_n)_n$ in $c_{00}$. Let $X$ be the completion of $c_{00}$ with respect to the norm $\no$ defined in \ref{norm}. 

	\begin{thm}\label{balls}
For every $x\in c_{00}$, $\|x\|_{F^*}=\| x\|$. Hence $X=F^*$.
	\end{thm}
	\begin{proof}
		Let $x\in c_{00}$ with $x=(x_i)_1^k$ for some $k$, we first prove that $\|x\|\le \|x\|_{F^*}$. We may assume that $\|x\|_{F^*}=1$, so $x\in B_{F^*}$. Using proposition \ref{prop_min}, it can be proved by induction that $V_n\subset B_X$ for all $n\in\N$. Therefore the convex hull of $V\cup\{0\}$ is contained in $B_X$. We know that  $x$ is the pointwise limit of a sequence $(y_n)_n$ in $\mathrm{conv}(V\cup \{0\})$. Since $V$ is closed under restrictions, we may assume that $\mathrm{supp}(y_n)\subset\mathrm{supp}(x)$. It follows that the sequence $(y_n)_n$ converges to $x$ with respect to the norm $\no$, and since $B_X$ is closed, we have that $\|x\|\le 1=\|x\|_{F^*}$.\\

	Now we check that $\|x\|_{F^*}\leq \|x\|$, again we may assume that $\|x\|=1$. 

	Let $(n_k)_k$ be a sequence of natural numbers such that $\pnorm{x}{n_k}<1+1/k$. Then $\frac{x}{1+1/k}\in B_{F^*}$ by the previous lemma, and since $\frac{x}{1+1/k}\longrightarrow x$ pointwise, it follows that $x\in B_{F^*}$.\\
	
	So we have proved that the norm $\no$ coincides with $\|\cdot\|_{F^*}$ on $c_{00}$. Since $c_{00}$ is dense in $X$ and $\langle e_n^*\rangle$ is dense in $F^*$, it follows that $X=F^*$, that is $F^*$ is the completion of $c_{00}$ with respect to the norm $\no$.

	\end{proof}
\section{A norm for the dual spaces of mixed \newline Tsirelson spaces}\label{mixed}
Let $\mathcal{M}$ denote a compact family (in the topology of pointwise convergence) of finite subsets of $\N$ which includes all singletons. We say that a family $E_1<\cdots<E_n$ of subsets of $\N$ is $\mathcal{M}$-admissible if there exists $M=\{m_i\}_1^n$ in $\mathcal{M}$ such that $m_1\leq E_1<m_2\leq E_2<\cdots<m_n\leq E_n$.\\

\begin{defi}
 	Let $(\mathcal{M}_n)_n, (\theta_n)_n$ be two sequences with each $\mathcal{M}_n$ a compact family of finite subsets of $\N$, $0<\theta_n<1$ and $\lim_n\theta_n=0$. The mixed Tsirelson space $\mathcal{T}[(\mathcal{M}_n,\theta_n)_n]$ is the completion of $c_{00}$ with respect to the norm
	\[
	 	\|x\|_*=\max\{\norm{x}{\infty}, \sup_n\sup\theta_n\sum_{i=1}^k\|E_ix\|_*\},
	\]
	where the inside sup is taken over all choices $E_1<E_2<\cdots<E_k$ of $\mathcal{M}_n$-admissible families.
\end{defi}
	
In this notation, Schlumprecht space $\mathcal{S}$ is the mixed Tsirelson space $\mathcal{T}[(\mathcal{A}_n,\frac{1}{\log_2(n+1)})_n]$, where $\mathcal{A}_n=\{F\subset\N:\#F\leq n\}$.\\

The definitions above and further properties of mixed Tsirelson spaces can be found in  of \cite[Part A, Chapter 1]{Argyros_Ramsey}. We shall present how our results can be adapted to give an expression for the norm of the dual of any mixed Tsirelson space $\mathcal{T}[(\mathcal{M}_n,\theta_n)_n]$. As in section \ref{norm_def}, we have the following definition.
	\begin{defi}\label{norm_mixed} For $x\in c_{00}$ let
	\begin{eqnarray*}
		\rho_0(x)&=&||x||_{\ell_1}\\
		\rho_{n+1}(x)&=& \min\{\min\{\frac{1}{\theta_l}\underset{1\leq i\leq k}{\max}\pnorm{E_ix}{n}:(E_i)_1^k \mbox{ }\mathcal{M}_l\mbox{-admissible},\\
		&&x=\sum_{i=1}^k E_i x, l\in\N\},\inf\{\rho_n(w^1)+\rho_n(w^2):x=w^1+w^2\}\}\\
		||x||&=&\lim_{n\to\infty}\rho_n(x)
	\end{eqnarray*}
	\end{defi}
	This defines a norm $||\cdot||$ on $c_{00}$, since the proof of lemma \ref{isnorm} can be carried out for this modified expression. Specifically, the argument for monotonicity is independent of the Shreier condition and of the coefficient 2, and clearly the properties (i)-(iii) used in the proof of the lemma are also satisfied.\\
	
	The proof of proposition \ref{prop_min} goes through as well, so the norm defined above satisifies the following implicit equation
	\begin{eqnarray*}
	||x||&=& \min\{\min\{\frac{1}{\theta_l}\underset{1\leq i\leq k}{\max}\norm{E_ix}{}: (E_i)_1^k \mbox{ }\mathcal{M}_l\mbox{-admissible}, x=\sum_{i=1}^k E_i x,l\in\N\},\\
		&&\inf\{\norm{y}{}+\norm{z}{}:x=y+z, \mathrm{supp} (y)\subseteq \mathrm{supp} (x)\}\},
	\end{eqnarray*}
	the standard basis of $c_{00}$ is normalized with respect to $\no$, and is the maximum such norm.\\
	To describe the unit ball of $\mathcal{T}[(\mathcal{M}_n,\theta_n)_n]^*$, we use the following sequence of sets:
	\begin{eqnarray*}
 		V_0 &=& \{\pm e_k^*: k\in\N\}\\
		V_{n+1} &=& V_n\cup \{\theta_l(f_1+\cdots +f_k): f_1<\cdots<f_k, f_i\in V_n, \\
	&&(\mathrm{supp}(f_i))_i^k,\mbox{ }\mathcal{M}_l\mbox{-admissible}, l\in\N\}\\
		V &=& \bigcup V_n.
	\end{eqnarray*}
	It is easy to see, just as for $F$, that for any $x\in \mathcal{T}[(\mathcal{M}_n,\theta_n)_n]$, 
	\[
	\|x\|_*=sup\{f(x):f\in V\}.
	\]
	So by the Bipolar Theorem, the unit ball of the dual of $\mathcal{T}[(\mathcal{M}_n,\theta_n)_n]$ is the weak*-closure of the convex hull of $V\cup\{0\}$. Since $(e_n)_n$ is a shrinking basis of $\mathcal{T}[(\mathcal{M}_n,\theta_n)_n]$, $(e_n^*)_n$ is a basis for $\mathcal{T}[(\mathcal{M}_n,\theta_n)_n]^*$. Hence one can define on $c_{00}$ the norm,
$$\|(a_n)_n\|_{\mathcal{T}[(\mathcal{M}_n,\theta_n)_n]^*}=\|\sum_{n}a_ne_n^*\|$$ for $(a_n)_n$ in $c_{00}$. Let $X$ be the completion of $c_{00}$ with respect to the norm $\no$ defined in \ref{norm_mixed}. We restate theorem \ref{balls} for $\mathcal{T}[(\mathcal{M}_n,\theta_n)_n]^*$:
	\begin{thm}
	 	For every $x\in c_{00}$, $\|x\|_{\mathcal{T}[(\mathcal{M}_n,\theta_n)_n]^*}=\| x\|$. Hence $X=\mathcal{T}[(\mathcal{M}_n,\theta_n)_n]^*$.
	\end{thm}
	For $x\in c_{00}$, the inequality $\|x\|\leq\|x\|_{\mathcal{T}[(\mathcal{M}_n,\theta_n)_n]^*}$ can be proved just as we did for $F^*$. For the reverse inequality, one can prove that the unit ball of $\mathcal{T}[(\mathcal{M}_n,\theta_n)_n]^*$ has the following properties:
	\begin{itemize}
 		\item[(i)] The sequence $(e_n^*)$ is contained in $B_{\mathcal{T}[(\mathcal{M}_n,\theta_n)_n]^*}$,
		\item[(ii)] if $f=(f_n)\in B_{\mathcal{T}[(\mathcal{M}_n,\theta_n)_n]^*}$ and $g=(g_n)$ is such that $|g_n|\leq |f_n|$ for all $n\in\N$, then $g\in B_{\mathcal{T}[(\mathcal{M}_n,\theta_n)_n]^*}$,
		\item[(iii)] for every $k\in\N$, if $f_1,\cdots,f_n\in B_{\mathcal{T}[(\mathcal{M}_n,\theta_n)_n]^*}\cap c_{00}$ are such that $f_1<\cdots<f_n$, and $(\mathrm{supp}(f_i))_i^n$ is $\mathcal{M}_k$-admissible, then $\theta_k(f_1+\cdots+f_n)\in B_{\mathcal{T}[(\mathcal{M}_n,\theta_n)_n]^*}\cap c_{00}$.
	\end{itemize}
	So the norm defined in \ref{norm_mixed} is the norm of the dual of the mixed Tsirelson space $\mathcal{T}[(\mathcal{M}_n,\theta_n)_n]$.\\
	
	At this moment it is not clear whether the norm of $T$ satisfies an implicit equation that allows us to calculate the norm of finitely supported vectors inductively in the cardinality of the support. A natural attempt is to replace the infimum term in proposition \ref{prop_min} by the $\ell_1$ norm but a few calculations show that the scalar function obtained does not satisfy the triangle inequality.\\
	
	\emph{Acknowledgments.} We thank Justin Moore and Jordi Lopez-Abad for their helpful suggestions. The author is grateful to Valentin Ferenczi for his comments related to the characterization of the norm.
	\bibliography{C:/Users/Diana/Documents/texfiles/mybib}
\end{document}